\documentclass[12pt,a4paper]{amsart}
\usepackage{amsmath,amssymb,amscd,amsthm}
\usepackage{graphics}
\usepackage{enumerate}
\input xy
\xyoption{all}

\usepackage{comment}
\usepackage{tikz}
\usepackage{setspace}

\usetikzlibrary{arrows}

\usepackage[active]{srcltx}
\usepackage{enumerate}
 \usepackage[colorlinks,final,backref=page,hyperindex]{hyperref}

\usepackage{amsmath,amssymb,xspace,amsthm}
\newtheorem{theorem}{Theorem}
\newtheorem{remark}[theorem]{Remark}
\newtheorem{lemma}[theorem]{Lemma}
\newtheorem{proposition}[theorem]{Proposition}
\newtheorem{corollary}[theorem]{Corollary}
\numberwithin{equation}{section}

\newcommand{\C}{\mathbb C}

\newcommand{\Z}{\mathbb{Z}}

\def\ml{\mathfrak{l}}

\def\mg{\mathfrak{g}}

\def\mm{\mathfrak{m}}
\def\mn{\mathfrak{n}}
\def\mh{\mathfrak{h}}
\def\mb{\mathfrak{b}}

\def\bs{\mathbf{s}}

\def\sl{\mathfrak{sl}}
\def\gl{\mathfrak{gl}}
\def\ml{\mathfrak{l}}

\def\ba{\mathbf{a}}
\def\br{\mathbf{r}}
\def\bs{\mathbf{s}}
\def\bm{\mathbf{m}}
\def\b1{\mathbf{1}}

\newcommand{\ch}{\mathcal{H}}
\newcommand{\cc}{\mathcal{C}}
\newcommand{\cd}{D}

\title[Whittaker category]{\bf   Whittaker  categories and  the minimal nilpotent finite $W$-algebras for  $\mathfrak{sl}_{n+1}$}
\author{Genqiang Liu, Yang Li}

\date{\today}

\begin{document}

\begin{abstract}
For any $\mathbf{a}=(a_1,\dots,a_n)\in \mathbb{C}^n$, we introduce  a Whittaker  category $\mathcal{H}_{\mathbf{a}}$
whose objects are $\mathfrak{sl}_{n+1}$-modules $M$ such that
  $e_{0i}-a_i$ acts locally nilpotently on $M$ for all $i \in \{1,\dots,n\}$,
and the subspace
 $\mathrm{wh}_{\mathbf{a}}(M)=\{v\in M \mid e_{0i} v=a_iv, \ i=1,\dots,n\}$ is finite dimensional.
In  this paper, we  first give a tensor product
 decomposition $U_S=W\otimes B$ of  the localization $U_S$ of $U(\mathfrak{sl}_{n+1})$ with respect to the Ore subset $S$ generated by $e_{01},\dots, e_{0n}$. We show that the associative algebra $W$ is isomorphic to
 the type $A_n$ finite $W$-algebra $W(e)$ defined by a minimal nilpotent element $e$ in $\mathfrak{sl}_{n+1}$.
Then using $W$-modules as a bridge, we show that
   each block with a generalized central character of $\mathcal{H}_{\mathbf{1}}$ is equivalent to the corresponding  block of the cuspidal
 category $\mathcal{C}$, which is completely characterized  by Grantcharov and Serganova.  As a consequence,
each regular integral block of $\mathcal{H}_{\mathbf{1}}$  and the category of
finite dimensional modules over $W(e)$
can be described by a well-studied quiver with certain quadratic relations.
\end{abstract}
\vspace{5mm}
\maketitle
\noindent{{\bf Keywords:}  Whittaker module, finite $W$-algebra, cuspidal module, quiver}
\vspace{2mm}

\noindent{{\bf Math. Subj. Class.} 2020: 17B05, 17B10, 17B20, 17B35}

\section{introduction}

Let $\mg$ be a complex semisimple Lie algebra. Whittaker modules and cuspidal modules are two classes of important modules for $\mg$. Let $\mn$ be the nilradical of a Borel subalgebra $\mb$ of $\mg$. Fix a Lie algebra homomorphism $\eta:\mn\rightarrow \C$.  A $\mg$-module $V$ is called a Whittaker module if $x-\eta(x)$ acts locally nilpotently on $V$ for any $x\in\mn$. A finitely generated weight $\mg$-module with finite dimensional weight spaces is called a cuspidal module if every root vector acts injectively on it.  Fernando \cite{F} proved that every simple weight module is cuspidal or is induced from some cuspidal module over a proper parabolic subalgebra, and further that the only simple Lie algebras which admit cuspidal modules are those of types $A$ or $C$. Simple cuspidal modules were classified by Mathieu, see \cite{M}. A complete classification and an explicit description of the  cuspidal category  for $\sl_{n+1}$ were given \cite{GS1}. All cuspidal generalized weight modules for $\sl_{n+1}$ were characterized in \cite{MS}.
The  cuspidal category for
$\mathfrak{sp}_{2n}$ is semisimple, see \cite{BKLM}.

The systematic study of Whittaker modules  for $\mg$ was started by Kostant, and several important results on Whittaker modules were proved in \cite{K}. The map $\eta$ is called non-singular if $\eta(x_\alpha)\neq 0$ for any simple root vector $x_\alpha$. When $\eta$ is non-singular,  Kostant has shown that the Whittaker category $\mathcal{N}(\eta)$ is equivalent to the category of modules over the  center of $U(\mg)$. Whittaker modules corresponding to singular $\eta$ can be defined similarly, see \cite{Ly,Mc}.  Whittaker modules can be studied in the context of category $\mathcal{O}$, Harish-Chandra bimodules and $D$-modules, see \cite{Ba,MS1,MS2}. Whittaker modules have been generalized to infinite dimensional Lie algebras with  triangular decompositions. For example, Whittaker modules over the Virasoro algebra and its generalizations were studied in \cite{OW1,OW2,GLZ, LPX,TWX}.
Whittaker modules have also been studied for generalized Weyl algebras by Benkart and Ondrus, see  \cite{BO}.  Whittaker modules for  affine
Lie algebras were studied  in \cite{ALZ, C,CF, CJ,GL,GZ}. In \cite{Pr1},  Premet has
shown that  the endomorphism algebras (called the finite $W$-algebras) of  universal Whittaker modules associated to  nilpotent elements are quantizations of  the coordinate rings of  special
transverse slices. The classification of finite dimensional irreducible modules over $W$-algebras was achieved in \cite{L,LO}.
The recent works in \cite{PT,T} focused on the description of one-dimensional $W$-modules.
In \cite{BM},  Batra and Mazorchuk  have defined  more
general notions of Whittaker modules.
They introduced  Whittaker pairs $(\mathfrak{g}, \mathfrak{n})$ of Lie algebras, where $\mathfrak{n}$ is a quasi-nilpotent Lie subalgebra of $\mathfrak{g}$ such that the adjoint action of $\mathfrak{n}$ on the quotient space $\mathfrak{g}/\mathfrak{n}$ is locally nilpotent. Under this general Whittaker set-up in \cite{BM}, they also determined a subcategory decomposition of the category $\widetilde{\mathcal{H}}$ consisting of $\mathfrak{g}$-modules such that $\mathfrak{n}$ acts locally finitely. A description of such category for $\mathfrak{sp}_{2n}$ was studied in \cite{LZZL}.
However, characterizations of the category $\widetilde{\mathcal{H}}$ for most Lie algebras are still open.

In this paper, we assume that $\mg=\sl_{n+1}$. The subspace $\mm_n=\oplus_{i=1}^n \C e_{0i}$  is a commutative subalgebra of $\sl_{n+1}$. Since the adjoint action of $\mm_n$ on the quotient space $ \sl_{n+1}/\mm_n$ is nilpotent,  $(\sl_{n+1},\mm_n)$ is a Whittaker pair.
For any $\mathbf{a}\in \C^n$, we have the   Whittaker category $\mathcal{H}_{\mathbf{a}}$ as in the abstract. In this paper, we study the block decompositions of $\mathcal{H}_{\mathbf{a}}$ for nonzero $\ba$, and find the relation between $\mathcal{H}_{\mathbf{a}}$ and the cuspidal  category $\mathcal{C}$.

 The paper is organized as follows. In Section 2, we collect all necessary preliminaries. In Section 3,
  we show that the subalgebra  $$W:=\{u\in U_S\mid [h_i, u]=[e_{0i},u]=0,\ 1\leq i\leq n\}$$ is a tensor product factor of
  $U_S$.
Furthermore, it is proven that  $W$  is  isomorphic to Premet's  type $A_n$ finite $W$-algebra defined by the  minimal nilpotent element $e=e_{10}+\dots+e_{n0}$, see Theorem \ref{iso-thw}.
 In Section 4,  we first give an explicit equivalence between $\ch_{\b1}$ and the category $W\text{-mod}$ of finite dimensional
 $W$-modules, see Proposition \ref{equ-hw}. Then we give realizations of all simple modules in
 $\ch_{\b1}$, see Theorem \ref{pi-image} and Proposition \ref{non-int-sing}, using the Shen-Larsson modules
 and finite dimensional simple modules  of finite $W$-algebras. After that, we show that  each
 block $W^\lambda\text{-mod}$ of $W\text{-mod}$ with the generalized central character $\chi_{\lambda}$ is equivalent
 to a block $\mathcal{C}_{\mu}^\lambda$ of the cuspidal category for a suitable $\mu$, see Theorem \ref{equ-cw}.
As a consequence,  $\ch_{1}^\lambda$ is equivalent to $\mathcal{C}_{\mu}^\lambda$, see Theorem \ref{main-th}.
It should be mentioned that both Theorem \ref{main-iso-th} and the morphism from
Proposition \ref{s-iso}  have some potential connection with the formulas on page $12$ given in the paper of Gelfand and Kirillov on skew-fields, see \cite{GK}.

In this paper, we denote by $\Z$ and $\C$ the sets of integers and complex numbers, respectively. All vector spaces and algebras are over $\C$. For a $z\in \Z$,  define $\Z_{\geq z}=\{m\in \Z\mid m\geq z\}$.
For a Lie algebra
$\mathfrak{g}$ we denote by $U(\mathfrak{g})$ its universal enveloping algebra. For an associative algebra $A$, let $Z(A)$ be its center.
 We write $\otimes$ for
$\otimes_{\mathbb{C}}$. For an $\alpha=(\alpha_1,\dots,\alpha_n)\in \C^n$, let $|\alpha|=\alpha_1+\dots+\alpha_n$.

\section{Preliminaries}

In this section, we collect some necessary definitions and results, including
Whittaker categories, cuspidal modules, central characters, translation  functors and Grantcharov-Serganova's Theorem on cuspidal modules.

\subsection{Whittaker category $\ch_\ba$ of $\mathfrak{sl}_{n+1}$}

 For a fixed  $n\in \Z_{\geq 2}$, let $\gl_{n+1}$ be the general linear Lie algebra over $\C$,  i.e., $\gl_{n+1}$ consists of $(n+1)\times (n+1)$-complex matrices.
 Let $e_{ij}$   denote the $(n+1)\times (n+1)$-matrix  unit whose $(i,j)$-entry is $1$ and $0$ elsewhere, $0\leq i,j \leq n$. Then $\{e_{ij}\mid 0\leq i,j \leq n\}$ is a basis of $\gl_{n+1}$. Denote $\mh_n=\oplus_{i=0}^{n-1}\C (e_{ii}-e_{i+1,i+1})$ which is a Cartan subalgebra of $\sl_{n+1}$.  In  $\mh_n$, let $h_i=e_{ii}-\frac{1}{n+1}(\sum_{k=0}^{n}e_{kk})$ for any $i\in \{0, 1,\dots,n\}$.
 Define a decomposition $\sl_{n+1}=\mm_n^-\oplus\ml_n\oplus \mm_n$, where $\mm_n^-=\oplus_{i=1}^n \C e_{i0}$, $\mm_n=\oplus_{i=1}^n \C e_{0i}$, and $\ml_n\cong \gl_n$ is the subalgebra spanned by $e_{ij}, 1\leq i \neq j\leq n$ and $h_k, k=1,\dots,n$. The subalgebra $\mathfrak{p}_n:=\ml_n\oplus \mm_n$ is a maximal parabolic subalgebra of $\sl_{n+1}$.

  Since $\mm_n$ is a commutative subalgebra of  $ \sl_{n+1}$ and the adjoint action of $\mm_n$ on the quotient space $ \sl_{n+1}/\mm_n$ is nilpotent,  $(\sl_{n+1},\mm_n)$ is a Whittaker pair in the sense of  \cite{BM}.
For an  $\sl_{n+1}$-module $M$,  we say
that the action of $\mm_n$  on M is locally finite provided that $U(\mm_n)v $ is finite dimensional for all $v\in M$.  Let $\widetilde{\ch}$ denote
the category consisting  of all $\sl_{n+1}$-modules with
 locally finite action of $\mm_n$. According to the action of
$\mm_n$, we have that
$$\widetilde{\ch}=\oplus_{\ba\in \C^n} \widetilde{\ch}_{\ba}, $$
where each $\widetilde{\ch}_{\ba}$ is a  full subcategory of $\widetilde{\ch}$
whose objects $M$ satisfying that  $e_{0i}-a_i$ acts locally nilpotently on $M$ for any $i \in \{1,\dots,n\}$.

 For any $\mathbf{a}=(a_1,\dots,a_n)\in \C^n$, we define a  full subcategory $\ch_\ba$ of $\widetilde{\ch}_{\ba}$
whose objects  $M$ satisfying that the subspace
 $$\mathrm{wh}_{\mathbf{a}}(M)=\{v\in M \mid e_{0i} v=a_iv, \ \forall\ i=1,\dots,n\}$$ is finite dimensional.
 An element in $\mathrm{wh}_{\mathbf{a}}(M)$ is called a Whittaker vector.

 \subsection{Cuspidal category $\mathcal{C}$ of $\sl_{n+1}$}

  An $\sl_{n+1}$-module  $M$ is called a {\it weight module} if $\mathfrak{h}_{n}$ acts diagonally on  $M$, i.e.
$$ M=\oplus_{\alpha\in \C^n} M_\alpha,$$
where $M_\alpha:=\{v\in M \mid h_iv=\alpha_i v, i=1,\dots,n\}.$   Denote $\mathrm{Supp}(M):=\{\alpha\in \C^n \mid M_\alpha\neq0\}$.

A weight
$\sl_{n+1}$-module $M$ is cuspidal if $M$ is finitely generated, all $M_{\alpha}$ are finite-dimensional, and $e_{ij}: M\rightarrow
M$ is  injective for all $i, j: 0\leq i \neq j \leq n$. Cuspidal modules play an important role in
the classification of all simple weight
modules with finite-dimensional weight spaces over reductive Lie algebras, see \cite{F,M}.
 Denote by $\mathcal{C}$ the category of all cuspidal
$\sl_{n+1}$-modules. A weight $\sl_{n+1}$-module $M$ is uniformly bounded if there is a $k\in \Z_{>0}$ such that $\dim M_{\alpha}\leq k$ for all $\alpha\in\mathrm{Supp}(M)$. Let $\mathcal{B}$ be  the category of all uniformly bounded $\sl_{n+1}$-modules. It is obvious that $\mathcal{C}\subset \mathcal{B}$.

  \subsection{Central characters}

Let $\{\epsilon_0,\dots, \epsilon_n\}$ be the dual basis  of $\{e_{00},\dots, e_{nn}\}$.
Then $$\mh_n^*=\{\lambda:=\sum_{i=0}^n \lambda_i\epsilon_i\mid \sum_{i=0}^n \lambda_i=0\}.$$
The half sum $\rho$ of the positive roots of $\sl_{n+1}$
is
$$\rho=\frac{1}{2}\sum_{0\leq i<j\leq n} (\epsilon_i-\epsilon_j)=\frac{1}{2}(n\epsilon_0+(n-2)\epsilon_1+\cdots-n\epsilon_{n}).$$
The {\it dot action} of the Weyl group $S_{n+1}$ of $\sl_{n+1}$ on $\mh_n^*$ is
defined by
$$ \omega\cdot\lambda=\omega(\lambda+\rho)-\rho.$$

A weight $\lambda$ is singular (resp. regular) if its stabilizer under the dot action of the Weyl group is non-trivial (resp. trivial). The following lemma is immediate.

\begin{lemma}\label{regular} \begin{enumerate}[$($a$)$]
 \item
 A weight $\lambda\in\mh_n^*$ is regular if and only if $\lambda_0,\lambda_1-1,\dots,\lambda_{n}-n$ are all distinct.
 \item For $1 \leq i\leq n$, $(0\cdots i)\cdot \lambda=(\lambda_i-i)\epsilon_0+\sum_{k=1}^i(\lambda_{k-1}+1)\epsilon_k+\sum_{l=i+1}^n\lambda_{l}\epsilon_l$.
 \end{enumerate}
\end{lemma}

A weight $\lambda\in\mh_n^*$ is called $\sl_{n+1}$-dominant
 if $\lambda_0-\lambda_1,\cdots, \lambda_{n-1}-\lambda_{n}\not\in \Z_{<0}$, and $\sl_{n+1}$-integral if
$\lambda_0-\lambda_1,\lambda_1-\lambda_2,\cdots, \lambda_{n-1}-\lambda_{n}\in \Z$.
We can see that an $\sl_{n+1}$-dominant integral weight is regular. A $\lambda\in\mh_n^*$  is called dot
dominant if $\lambda+\rho$ is $\sl_{n+1}$-dominant.

Denote by $Z(\sl_{n+1})$ the center of $U(\sl_{n+1})$. An algebra homomorphism $\chi: Z(\sl_{n+1})\rightarrow \C$ is called a central character.
Let $\Theta$ be the set of all central characters of $\sl_{n+1}$.  We have a map $ \xi : \mh_n^*\rightarrow \Theta $ which maps $\mu \in \mh_n^*$ to the central character $\chi_{\mu}$ associated with the Verma module $M(\mu)$, since $\dim M(\mu)_{\mu}=1$.
By the Harish-Chandra's Theorem, the map $\xi $ is surjective, see \cite{HC}. Moreover $\chi_{\mu}\cong \chi_{\lambda}$ if and only if $\mu=\omega\cdot \lambda$ for some $\omega\in S_{n+1}$.  An $\sl_{n+1}$-module $M$ is said to  have the generalized central character $\chi$ if  $z-\chi(z)$ acts locally nilpotently on $M$ for any $z\in Z(\sl_{n+1})$.
Let $\mathcal{M}$ denote the category of finitely generated $U(\sl_{n+1})$ -modules such that the action of $Z(\sl_{n+1})$ is locally finite. For each $ \lambda\in\mh_n^*$, we use  $\mathcal{M}^{\lambda}$ to denote the full subcategory of $\mathcal{M}$ consisting of $\sl_{n+1}$-modules which have the generalized central character $\chi_{\lambda}$. We can see that  $\mathcal{M}^{\lambda}=\mathcal{M}^{\mu}$ if and only if $\mu=\omega\cdot \lambda$ for some $\omega\in S_{n+1}$.
In particular, when $\lambda=0$, the block $\mathcal{M}^{0}$ is called the principal block of $\mathcal{M}$.
 For any $\lambda\in \mh_n^*$ and any  full subcategory $\mathcal{N}$ of $\mathcal{M}$, we denote $\mathcal{N}^{\lambda}=\mathcal{N}\cap \mathcal{M}^{\lambda}$. For example,
$\ch_{\ba}^{\lambda}=\ch_\ba\cap \mathcal{M}^{\lambda}$, $\cc^{\lambda}=\cc \cap \mathcal{M}^{\lambda}$, $\mathcal{B}^{\lambda}=\mathcal{B} \cap \mathcal{M}^{\lambda}$.

\subsection{Translation functor}
 Now let us recall the translation functors in \cite{BG}. Suppose that $V$ is  a finite dimensional simple $\sl_{n+1}$-module, and $\mu,\lambda\in\mh_n^*$
satisfying that $\lambda-\mu\in \text{Supp} V$. The translation
functor $$T^{\mu,\lambda}:\mathcal{M}^{\mu}\rightarrow \mathcal{M}^{\lambda},$$
is defined by
$$T^{\mu,\lambda}(M)=(M\otimes V)^{\chi_\lambda},$$ where
$(M\otimes V)^{\chi_\lambda}$ is the direct summand of $M \otimes V$ having the generalized
 central character $\chi_\lambda$. If $\lambda, \mu$ satisfy the
 following three conditions:
 \begin{enumerate}[$($a$)$]
\item $\lambda-\mu$ belongs to the $S_{n+1}$-orbit of the highest weight of $V$;
\item  the
stabilizers of $\mu+\rho$ and $\lambda+\rho$ in the Weyl group coincide;
\item  $\mu+\rho$ and $\lambda+\rho$ lie in the same
Weyl chamber,
 \end{enumerate} then
 $T_V^{\mu,\lambda}:\mathcal{M}^{\mu}\rightarrow \mathcal{M}^{\lambda}$
is an equivalence, see \cite{BG}. We
say that two weights $\lambda,\mu\in \mh_n^*$ are in the same Weyl chamber if $\lambda-\mu$ is $\sl_{n+1}$-integral, and for any positive root $\alpha$ of $\sl_{n+1}$ such that $\lambda(h_\alpha)\in \Z$, $\lambda(h_\alpha)\in \Z_{\geq 0}$ if and only if $\mu(h_\alpha)\in\Z_{\geq 0}$, where $h_\alpha\in\mh_n$ is the dual root of $\alpha$.

 \subsection{Grantcharov-Serganova's Theorem on category $\mathcal{C}$}
For  $u\in\C^{n}$,
let $\mathcal{C}_{u}$ (resp. $\mathcal{B}_{u}$) be the full subcategory of $\mathcal{C}$ (resp. $\mathcal{B}$) consisting of modules $M$ such that $\mathrm{Supp}(M)\subseteq u+\Z^n$.  Let $\mathcal{C}_{u}^{\lambda}= \mathcal{C}_{u}\cap \mathcal{M}^{\lambda}$ and $\mathcal{B}_{u}^{\lambda}= \mathcal{B}_{u}\cap \mathcal{M}^{\lambda}$ for any $\lambda\in \mh_n^*$.

 We use $\gamma$ to denote the projection $(\oplus_{i=0}^n \C e_{ii})^* \rightarrow \mh_n^*$
with the one-dimensional kernel $\sum_{i=0}^n \epsilon_i$. For example $\gamma(c\epsilon_0)=\frac{nc}{n+1}\epsilon_0-\frac{c}{n+1}(\sum_{i=1}^n \epsilon_i)$. It is immediate that $\gamma(c\epsilon_0)$ is singular integral if and only if $c\in \{-1,\dots,-n\}$.
The following lemma was proven in \cite{GS1}.

\begin{lemma}If the category $\mathcal{C}^\lambda$ is not empty, then it is equivalent to $\mathcal{C}^{\gamma(c\epsilon_0)}$ for some $c\in (\C\setminus \Z)\cup
\{0,-1,\dots,-n\}$.
\end{lemma}

For the cuspidal category $\mathcal{C}$ over  $\mathfrak{sl}_{n+1}$, the translation functor
can not provide equivalence of the subcategories ${\mathcal
C}^{\lambda}$ of ${\mathcal C}$ for all different $\chi_{\lambda}$.
There are  three essentially different
central character types: nonintegral, regular integral and
singular integral.

 \begin{theorem}\label{GS-th}\cite{GS1} Let $\lambda\in \mh_n^*$ and $u\in\C^n$ such that $\mathcal{C}^{\lambda}_{u}$ is not empty.
 \begin{enumerate}[$($a$)$]
 \item If $\lambda$ is singular or non-integral, then $\mathcal{C}^{\lambda}_{u}$ is equivalent to the category of
finite-dimensional modules over the algebra $\C[[x]]$.
 \item    If $\lambda$ is regular integral,  then $\mathcal{C}^{\lambda}_{u}$ is  equivalent to the category of finite-dimensional locally nilpotent modules over the quiver
$$
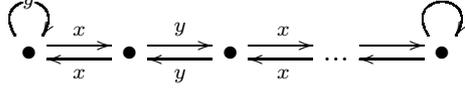
\xymatrix{\bullet \ar@(ul,ur)[]|{y} \ar@<0.5ex>[r]^x & \bullet
\ar@<0.5ex>[l]^x \ar@<0.5ex>[r]^y & \bullet \ar@<0.5ex>[l]^y
\ar@<0.5ex>[r]^x & \ar@<0.5ex>[l]^x... \ar@<0.5ex>[r] & \bullet
\ar@<0.5ex>[l] \ar@(ul,ur)[]|{}}
$$
containing $n$ vertices with relations $xy=yx=0$.
   \end{enumerate}
 \end{theorem}

The above quiver has tame representation type, see \cite{Er}. It was
originally used by
Gelfand--Ponomarev in \cite{GP}  to classify
indecomposable representations of the Lorentz group.

\begin{remark}When $\lambda$ is a non-integral or a
singular integral weight, Corollary 2.6 in \cite{GS1} gives the sufficient and necessary conditions for  non emptiness of $\mathcal{C}^{\lambda}_{u}$. From this corollary, we can see that $\mathcal{C}^{\gamma(c\epsilon_0)}_{u}$ is nonempty if and only if  $$u_1+\frac{c}{n+1}, \cdots, u_n+\frac{c}{n+1}, -|u|+\frac{c}{n+1}\not\in \Z,$$  when $c\in (\C\setminus \Z)\cup
\{-1,\dots,-n\}$. By Corollary 6.14 in \cite{GS1},  if  $u_1, \dots, u_n, -|u|\not\in \Z$, then $\mathcal{C}^{0}_{u}$ is nonempty.
\end{remark}

\section{A tensor product decomposition of the localized enveloping algebra $U_S$}

  Throughout the paper, we denote $U=U(\sl_{n+1})$ and $I_{n+1}=\sum_{i=0}^n e_{ii}$.
 Let  $U_S$ be the localization of $U$ with respect to multiplicative subset $S$ generated by $e_{01},\dots, e_{0n}$. In this section, we give a tensor product decomposition of  $U_S$, which is very useful for the study of $\sl_{n+1}$-modules with bijective actions of $e_{01},\dots, e_{0n}$.

\subsection{The category  $\ch_{\b1}$}

For an $(n+1)\times (n+1)$ elementary matrix $S$, let $\sigma_S$ be the conjugate automorphism of $\sl_{n+1}$ such that $X\mapsto S^{-1}XS$ for all $X\in\sl_{n+1}$. For an $\sl_{n+1}$-module $M$, and $\sigma\in \text{Aut}(\sl_{n+1})$, $M$ can be twisted by $\sigma$ to  be a new $\sl_{n+1}$-module $M^\sigma$. As a vector space $M^\sigma=M$, whose  $\sl_{n+1}$-module
structure is defined by $X\cdot v= \sigma(X)v, \forall\ X\in \sl_{n+1}, v\in M$. We denote the vector $(1, 1,\dots, 1)\in \C^n$ by $\b1$.

\begin{lemma}\label{nonzero} If $\ba\in \C^n$ is nonzero,  then $\ch_{\ba}$ is equivalent to
$\ch_{\b1}$.
\end{lemma}
\begin{proof}
We twist modules in  $\ch_{\ba}$ using   conjugate automorphisms of $\sl_{n+1}$ given by elementary matrices several times.  After  suitable permutations of rows and columns, we can assume that $a_1\neq 0$. Using the isomorphism of $\sl_{n+1}$ mapping  $e_{01}$ to $\frac{1}{a_1}e_{01}$, we can suppose that $a_{1}=1$.
Adding a proper  multiple of the first column to the  $i$-th column, we can assume that
$a_i=1$ for any $2\leq i\leq n$.
\end{proof}

 From now on, we always assume that
$\ba=\b1$.

The following lemma is a key observation on modules in $\ch_{\mathbf{1}}$.

\begin{lemma}\label{free}
Any module $M$ in
$\ch_{\mathbf{\b1}}$ is a free  $U(\mh_n)$-module of finite rank, and as a vector space,
$M\cong U(\mh_n)\otimes \mathrm{wh}_{\mathbf{1}}(M)$.\end{lemma}

\begin{proof}For an $\bm\in \Z_{\geq 0}^n$, define
$$ e_{\bm}=(e_{01}-1)^{m_1}\cdots(e_{0n}-1)^{m_n},\ \ h^{\bm}=h_1^{m_1}\cdots h_n^{m_n}.$$
Note that the set $\{e_{\bm}| \bm\in\Z_{\geq 0}^n\}$ forms a basis for $U(\mm_n)$. We define the total order on $\Z_{\geq 0}^n$ such that: $\br<\bm$ if $|\br|<|\bm|$ or
$|\br|=|\bm|$ and
there is an $l\in \{1,\cdots, n\}$ such that $r_i=m_i$ when $1\leq i <l$ and $r_l< m_l$. For each $\bm$, the set $\{\br \in \Z_{\geq 0}^n\mid \br< \bm\}$ is a finite set.
For a nonzero $\bm \in \Z_{\geq 0}^n$, denote by $\bm'$ the predecessor
of $\bm$, that is, $\bm'<\bm$ and there is no $\bs\in \Z_{\geq 0}^n$ such that $\bm'<\bs<\bm$.

From $[h_i,e_{0i}]=-e_{0i}$, we have that $$e_{0i}h_i^{m_i}=(h_i+1)^{m_i}e_{0i},\ \ [e_{0i},h_i^{m_i}]=\sum_{j=1}^{m_i}\binom{m_i}{j}h_i^{m_i-j}e_{0i}.$$
Then by induction on $m_i$,  we can show that for any  nonzero $v\in \text{wh}_{\b1}(M)$,
$$(e_{0i}-1)^{m_i}h_i^{m_i}v=k_{m_i}v,\ \ (e_{0i}-1)^{s_i}h_i^{m_i}v=0,\  s_i>m_i, $$ for some nonzero $k_{m_i}\in \C$.
 Consequently  for any $\bm,\bs\in \Z_{\geq 0}^n$  and nonzero $v\in \text{wh}_{\b1}(M)$, we have that $e_{\bs} h^{\bm}v=0$ whenever $\bs>\bm$,
$e_{\bm} h^{\bm} v= k_{\bm} v$ for some nonzero scalar $k_{\bm}$.

For each $\bm \in \Z_{\geq 0}^n$,  denote by $\mathbb{I}_{\bm}$  the ideal of $U(\mm_n)$ spanned by  $e_{\bs}$ with $\bs>\bm$, and $M_{\bm}=\{w\in M\mid \mathbb{I}_{\bm}w=0\}$. It is clear that $\text{wh}_{\b1}(M)=M_{\mathbf{0}}$.
For any nonzero $w\in M$, since $M\in \mathcal{H}_{\b1}$, there is an
$\bm \in \Z_{\geq 0}^n$ such that $w\in M_{\bm}\setminus M_{\bm'}$, i.e.,
$e_{\bm} w\neq 0$ and $e_{\bs} w=0$ for any $\bs>\bm$. So $e_{\bm} w\in \text{wh}_{\b1}(M)$. By the above discussion, $e_{\bm} h^{\bm} e_{\bm} w=k_{\bm}e_{\bm} w  $.
We call $\bm$ the degree of $w$.
We use  induction on the degree  $\bm$ of $w$ to show that $w\in U(\mh_n)\text{wh}_{\b1}(M)$. Let $w'=w-\frac{1}{k_{\bm}}h^{\bm}e_{\bm}w$. Then
$$e_{\bm} w'=e_{\bm} w-\frac{1}{k_{\bm}}e_{\bm} h^{\bm}e_{\bm}w=0. $$
This implies that the degree of $w'$ is smaller than $\bm$. By the induction
hypothesis, $w' \in U(\mh_n)\text{wh}_{\b1}(M)$. Hence $w\in U(\mh_n)\text{wh}_{\b1}(M)$.

Finally we show that $M\cong U(\mh_n)\otimes \mathrm{wh}_{\mathbf{1}}(M)$. Let $\{v_i| i\in \Lambda\}$ be a basis of
$\mathrm{wh}_{\mathbf{1}}(M)$.
Suppose that
$$w:=\sum_{\mathbf{r} \leq \mathbf{m}} \sum_{i\in\Lambda} c_{\mathbf{r},i}h^{\mathbf{r}} v_i=0$$ in $M$, where $c_{\mathbf{r},i}\in \C$.
From
$e_{\bm} w=0$, we see that $c_{\bm,i}=0$. Then  by induction on $\bm$, we have that
$c_{\br,i}=0$ for any $\br < \bm$ and $i$. Thus  $\{ h^{\mathbf{m}}v_i\mid \mathbf{m}\in \Z_{\geq 0}^n, i\in \Lambda\}$ is linearly independent. In conclusion, we complete the proof.
\end{proof}

We will use the theory of finite $W$-algebra to study $\ch_{\mathbf{\b1}}$. If we denote
$$\mg(0)=\ml_n,\ \  \mg(2)=\mm_n^-, \ \ \mg(-2)=\mm_n,$$ then
$$\sl_{n+1}=\mg(-2)\oplus\mg(0)\oplus \mg(2)$$ is a good $\Z$-grading
for the minimal nilpotent element $e=e_{10}+\dots+e_{n0}$, see \cite{EK}. That is $e\in \mg(2)$,
$\text{ad}e: \mg(0)\rightarrow\mg(2)$ is surjective, and $\text{ad}e: \mg(-2)\rightarrow\mg(0)$ is injective. The map $\theta:\mm_n \rightarrow \C, X\mapsto \text{tr}(Xe)$ defines a one dimensional
$\mm_n$-module $\C_{\b1}:=\C v_{\b1}$. It is clear that $\theta(e_{0i})=1$,  for any $1\leq i \leq n$.
Define  the induced $\sl_{n+1}$-module
(called generalized Gelfand-Graev module)
$$Q_{\b1} := U(\sl_{n+1}) \otimes_{U(\mm_n)}\C_{\b1}.$$

The finite $W$-algebra $ W(e)$ is defined to be
the endomorphism algebra
$$W(e) := \text{End}_{\sl_{n+1}}(Q_{\b1} )^{\text{op}}. $$
This is the original definition of finite $W$ algebras defined by Premet, see \cite{Pr1, BK} and \cite{W}. By the Skryabin's equivalence \cite{Pr1},
the functors $$M\mapsto \text{wh}_{\b1}(M),
\ \ \ \  V\mapsto Q_{\b1}\otimes_{W(e)} V, $$ are inverse equivalence between  $\ch_{\b1}$ and the category of finite dimensional $W(e)$-modules.

\subsection{An explicit realization of the finite $W$-algebra $ W(e)$}
In this subsection, we  give an explicit realization of the finite $W$-algebra $ W(e)$.

Since each $\text{ad}e_{0i}$ is locally nilpotent on $U$, the set
   $$S:=\{e_{01}^{i_1}\dots e_{0n}^{i_n}\mid i_1,\dots, i_n\in\Z_{\geq0}\}$$ is an Ore subset of $U$, see Lemma 4.2 in \cite{M}.
   Denote
 by $U_{S}$  the localization $U$ with respect to $S$.
 Mathieu has used $U_{S}$  to study simple cuspidal modules.
 We will give a tensor product decomposition of  $U_{S}$.

 Let $W:=\{u\in U_S\mid [\mh_n, u]=[\mm_n,u]=0\}$ which is a subalgebra of $U_S$.
 Define the following elements in $U_S$:
 \begin{equation}\label{x-w-def}
 \aligned x_{ij}&=e_{ij}e_{0i}e_{0j}^{-1}-h_i,\\ \omega_k&=e_{k0}e_{0k}+\sum_{j=1}^n
 (e _{kj}-\frac{\delta_{jk}}{n+1}I_{n+1})(h_j-1)e_{0k}e_{0j}^{-1}, \endaligned\end{equation}for all $i,j, k=1,\dots,n$ with
 $i\neq j$. In the following lemma, we show that these elements belong to $W$.

 \begin{lemma}\label{comm}
 For any $i,j, k\in\{1,\dots,n\}$, $x_{ij}, \omega_k\in W$.
 \end{lemma}

 \begin{proof} First, it can be seen that $$\aligned \ [h_i, \omega_k]=0,\ \ [h_k, x_{ij}]=[e_{kk}, e_{ij}e_{0i}e_{0j}^{-1}-e_{ii}]=0,
 \endaligned$$ and
 $$\aligned \ [e_{0k}, x_{ij}]&=[e_{0k}, e_{ij}e_{0i}e_{0j}^{-1}-e_{ii}]
 =\delta_{ki}e_{0j}e_{0i}e_{0j}^{-1} -\delta_{ki}e_{0i}
 =0,
 \endaligned$$ for any $k\in\{1,\dots,n\}$. So $x_{ij}\in W$.

 We can also compute that
 $$\aligned \   e_{0q}\omega_k& = e_{0q}\Big( e_{k0}e_{0k}
 +\sum_{j=1}^ne_{kj}(h_j-1)e_{0k}e_{0j}^{-1}-\frac{1}{n+1}I_{n+1}(h_k-1)\Big)\\
&= \delta_{qk}e_{00}e_{0k}-e_{kq}e_{0k}+e_{k0}e_{0k}e_{0q}
+\delta_{qk}\sum_{j=1}^ne_{0j}(h_j-1)e_{0k}e_{0j}^{-1}\\
&\ \ \ \ +\sum_{j=1\atop j\neq q}^ne_{kj}(h_j-1)e_{0k}e_{0j}^{-1}e_{0q}+ e_{kq}h_qe_{0k}-\frac{1}{n+1}I_{n+1}(h_k-1+\delta_{qk})e_{0q}.
\endaligned$$

From
$$\aligned \omega_ke_{0q}=e_{k0}e_{0k}e_{0q}
 +\sum_{j=1}^ne_{kj}(h_j-1)e_{0k}e_{0j}^{-1}e_{0q}-\frac{1}{n+1}I_{n+1}(h_k-1)e_{0q},
\endaligned$$ it can be checked that
$$\aligned \   e_{0q}\omega_k- \omega_ke_{0q}
&= \delta_{qk}e_{00}e_{0k}-e_{kq}e_{0k}
+\delta_{qk}\sum_{j=1}^ne_{0j}(h_j-1)e_{0k}e_{0j}^{-1}\\
&\ \ \ \ -e_{kq}(h_q-1)e_{0k}+ e_{kq}h_qe_{0k}-\delta_{qk}\frac{1}{n+1}I_{n+1}e_{0q}\\
&=\delta_{qk}e_{00}e_{0k}
+\delta_{qk}\sum_{j=1}^nh_je_{0k}-\delta_{qk}\frac{1}{n+1}I_{n+1}e_{0k}\\
&= \delta_{qk}\sum_{j=0}^nh_je_{0k}=0.
\endaligned$$
The proof is completed.
 \end{proof}

 Let $B$ be the subalgebra of $U_{S}$ generated by $h_{i}, e_{0k}, e_{0k}^{-1}, 1\leq k,i\leq n$. In the next theorem, we show that $W$ is a tensor factor of the localized enveloping algebra
$U_{S}$.

 \begin{theorem}\label{main-iso-th}  \begin{enumerate}[$($a$)$]
 \item $U_{S}\cong W\otimes  B$, $Z(U_{S})\cong Z(W)$.
 \item All the ordered monomials in $x_{ij}, \omega_k $, $i, j,k=1,\dots, n $ with
  $i\neq j$ form a basis of $W$
over $\C$.
   \end{enumerate}
\end{theorem}

\begin{proof}(a)
  Since $[W, B]=0$, by (\ref{x-w-def}),
we have an algebra homomorphism
  $$\tau: W\otimes  B\rightarrow U_{S},$$ such that
  $$\aligned \tau(x_{ij})&=e_{ij}e_{0i}e_{0j}^{-1}-h_i,\\ \tau(\omega_k)&=e_{k0}e_{0k}+\sum_{j=1}^n
 (e _{kj}-\frac{\delta_{jk}}{n+1}I_{n+1})(h_j-1)e_{0k}e_{0j}^{-1},
  \endaligned$$
and $\tau|_{B}=\text{id}_{B}$.

From
\begin{equation}\label{ws-relation}\aligned\ e_{ij}& =x_{ij}e_{0j}e_{0i}^{-1}+h_ie_{0j}e_{0i}^{-1}\in WB, \ \ i\neq j,\\
  e_{k0}&= \omega_ke_{0k}^{-1}-\sum_{j=1}^n
  (e _{kj}-\frac{\delta_{jk}}{n+1}I_{n+1})(h_{j}-1)e_{0j}^{-1}\in WB,
  \endaligned\end{equation}
we see that $\tau$ is surjective. Note that $\tau$ maps a polynomial
 in $x_{ij}, \omega_k$ to a polynomial in $e_{ij}, e_{k0}, e_{0k}, h_l$.  By the PBW Theorem on $U_S$, examining  the highest degree term containing $e_{ij}, e_{k0}$ in $\tau(u)$ for any polynomial $u$
 in $x_{ij}, \omega_k$,  we can obtain that
 $\tau|_W$ is injective. Therefore  $\tau$ is injective.

   The second assertion follows from  that $Z(U_{S})\cong Z(B)\otimes Z(W)$ and $Z(B)=\C$.

  (b) By (\ref{ws-relation}) and the PBW Theorem on $U_S$, the claim of (b) follows.
\end{proof}

 \begin{theorem}\label{iso-thw} We have the algebra isomorphism $W\cong W(e)$.
\end{theorem}

\begin{proof}
First, we can see that  $$\{e_{ij}-h_i, e_{k0}\mid 1\leq i, j,k \leq n, i\neq j \}$$
  is a basis of the centralizer of $e$ in  $\sl_{n+1}$.
  By Lemma \ref{comm},
$$\aligned x_{ij}v_{\b1}&=(e_{ij}-h_i)v_{\b1}, \ i\neq j \\
 \omega_k v_{\b1}& =(e_{k0}+\sum_{j=1}^ne_{kj}(h_{j}-1))v_{\b1}, \endaligned$$
 are Whittaker vectors of the module $Q_{\b1}$. So there are $\Theta_{i,j}, \Theta_k \in \text{End}_{\sl_{n+1}}(Q_{\b1} )$
 such that $\Theta_{i,j}(v_{\b1})= x_{ij}v_{\b1}$ and $ \Theta_k(v_{\b1})=\omega_k v_{\b1}$. By Theorem 4.6 in \cite{Pr1}, the monomials in $\Theta_{i,j}, \Theta_k$, $1\leq i, j,k \leq n$ with $i\neq j$ form a basis of $W(e)$. By (b) in Theorem \ref{main-iso-th}, the map $W\rightarrow  W(e)$ such that
 $$x_{ij}\mapsto \Theta_{i,j}, \ \ \omega_k\mapsto \Theta_k, $$
 defines an isomorphism.
\end{proof}

 \section{Characterizations of all blocks of  $\ch_{\b1}$}

Let $W\text{-mod}$ be the category of finite dimensional $W$-modules. In this section, using  $W$-modules, we will classify simple objects in $\ch$ and
find equivalent relations between $\ch$ and $\mathcal{C}$.

\subsection{$\ch_{\b1}$ and $W$-modules}
By the Skryabin's equivalence \cite{Pr1} and the isomorphism between
 $W$ and $W(e)$, the category $\ch_{\b1}$ is equivalent to
 $W\text{-mod}$. Next we construct explicit equivalent functor
 between $\ch_{\b1}$ and
 $W\text{-mod}$.

We define the functor $F: \ch_{\b1}\rightarrow W\text{-mod}$ such that $F(M)=\text{wh}_{\b1}(M)$. From $[W, e_{0i}]=0$ and that
$\text{wh}_{\b1}(M)$ is finite dimensional,
 $F(M)$ is  a finite dimensional $W$-module. Conversely, for  a $V\in W\text{-mod}$, let each $e_{0i}$
 act  identically on $V$.
 Let $$G(V)=\text{Ind}_{U(\mm_n)_SW}^{U_S}V=U_{S}\otimes_{U(\mm_n)_SW} V.$$ It is clear that $G(V)=U(\mh_n)\otimes V$.
 From  (\ref{ws-relation}),
  the  $\sl_{n+1}$-module structure on $G(V)$ satisfies that:
 $$\aligned \ h_i \cdot (f(h)\otimes v)&= (h_if(h))\otimes v,\\
 e_{0i}\cdot (f(h)\otimes v)&=  f(h+e_i)\otimes v,
  \endaligned$$ and
 $$\aligned
 e_{ij}\cdot (f(h)\otimes v)&=f(h-e_i+e_j)\otimes x_{ij}v
 +h_if(h-e_i+e_j)\otimes v,  \\
  e_{k0}\cdot (f(h)\otimes v)&=f(h-e_k)\otimes \omega_k v -\sum_{j=1}^n
   (e _{kj}-\frac{\delta_{jk}}{n+1}I_{n+1})\cdot (h_j-1)f(h-e_j)\otimes v,
 \endaligned$$
 where $i, j, k \in \{1,\dots, n\}$ with $i\neq j$, $f(h)\in U(\mh_n)$
 and $v\in V, f(h+e_i)=f(h_1,\dots, h_i+1,\dots, h_n)$.
Then we obtain a functor $G: W\text{-mod} \rightarrow \ch_{\b1}$ such that $G(V)=U(\mh_n)\otimes V$.

\begin{proposition}\label{equ-hw} We have that $FG=\text{id}_{W\text{-mod} }$ and $GF=\text{id}_{\ch_{\b1} }$.
 So $\ch_{\b1}$ is equivalent to $W\text{-mod}$.
\end{proposition}

\begin{proof}
 Since $FG(V)=\text{wh}_{\mathbf{\b1}}(U(\mh_n)\otimes V)\cong V$ for any $V\in W\text{-mod} $,
 we have $FG\cong\text{id}_{W\text{-mod} }$.
For any $M\in \ch_{\mathbf{1}}$, by Lemma \ref{free}, as a vector space, $M\cong U(\mh_n)\otimes \mathrm{wh}_{\b1}(M)$. By the formula (\ref{ws-relation}) and the isomorphism $U_{S}\cong W\otimes  B$,
we have that $M\cong G(\mathrm{wh}_{\b1}(M))\cong GF(M)$. So  $GF\cong\text{id}_{\ch_{\b1} }$.
\end{proof}

By Theorem $\ref{main-iso-th}$, $Z(U_S)$
 can be canonically identified
with the center $Z(W)$ of $W$. For any $\lambda\in\mh_n^*$,
let $W^\lambda\text{-mod}$ be the full subcategory of $W\text{-mod}$ consisting of modules having the generalized central character $\chi_{\lambda}$. By Proposition \ref{equ-hw}, we the following corollary.

\begin{corollary}\label{p7} For any $\lambda\in\mh_n^*$, $\ch^\lambda_{\b1}$ is equivalent to $W^\lambda\text{-mod}$.
\end{corollary}

\subsection{Blocks of $\ch_{\b1}$}

In this subsection, we will study equivalences between different blocks of $\ch_{\b1}$ using the translation functors.
We first recall the weighting functor introduced in \cite{N16}.
For a point $b\in\C^n$,  let $I_b=\langle h_1-b_1,\dots, h_n-b_n\rangle$ be the maximal ideal of  $U(\mh_n)=\C[h_1,\dots,h_n]$ corresponding to $b$.
For an $\sl_{n+1}$-module $M$ and  $b \in\C^n$, set $M^{b }:= M/I_{b}M$. For  $u=(u_1,\dots,u_n)\in \C^n$,
let  $$\mathfrak{W}^u(M):=\bigoplus_{b\in\Z^n}M^{b+u}.$$
The space $\mathfrak{W}^u(M)$ becomes  a weight
 module   under the following action:
\begin{equation}\label{3.3}
X_{\alpha} \cdot(v+I_{b}M):= X_{\alpha} v+I_{b+\alpha}M, v\in M, \alpha\in \Delta, b\in u+\Z^n,
\end{equation}
see Proposition 8 in \cite{N16}, where the root $\alpha$  is identified with
 $(\alpha(h_1),\dots,\alpha(h_n))$. We can see that $\text{supp}( \mathfrak{W}^u(M))\subset  u+\Z^n$.

\begin{lemma}\label{equ-lemm}Suppose that $\lambda=\sum_{i=0}^n \lambda_i\epsilon_i\in \mh_n^*$.
\begin{enumerate}
\item If  $\ch_{\b1}^{\lambda}$ is nonempty, then we can assume that
$\lambda_1-\lambda_2,\dots, \lambda_{n-1}-\lambda_n\in \Z_{\geq 0}$.
\item If $\lambda$ is not $\sl_{n+1}$-integral, then $\ch_{\b1}^{\lambda}$
is equivalent to $\ch_{\b1}^{c\gamma(\epsilon_0)}$ for some
$c\in \C$ with $c\not\in \Z$.
\item If $\lambda$ is $\sl_{n+1}$-integral and singular, then $\ch_{\b1}^{\lambda}$
is equivalent to $\ch_{\b1}^{-k\gamma(\epsilon_0)}$ for some
$k\in\{1,\dots,n\}$.
\item If $\lambda$ is $\sl_{n+1}$-integral and regular, then $\ch_{\b1}^{\lambda}$
is equivalent to $\ch_{\b1}^{0}$.
\end{enumerate}
\end{lemma}
\begin{proof} (1)
By Lemma \ref{free}, any module $M$ in $\ch_{\b1}^{\lambda}$ is a free $U(\mh_n)$-module of finite rank. Choose a $u\in \C^n$ such that $\mathcal{B}^{\lambda}_{u}$ is not empty.
Then   $\mathfrak{W}^u(M)$ is a uniformly bounded
weight module, i.e., $\mathfrak{W}^u(M)\in\mathcal{B}^{\lambda}_{u}$.  By the structure of uniformly bounded
weight modules, there is a $\nu$ in the orbit $S_{n+1}\cdot \lambda$ such that $\nu_1-\nu_2,\dots, \nu_{n-1}-\nu_n\in \Z_{\geq 0}$, see Proposition 8.5 in \cite{M}.

(2) By (1), we can assume that $\lambda_1-\lambda_2,\dots, \lambda_{n-1}-\lambda_n\in \Z_{\geq 0}$.
If $\lambda$ is not $\sl_{n+1}$-integral, then
$-\lambda_1-\sum_{i=1}^n \lambda_i\not \in \Z$.   Set $c=-\lambda_1-\sum_{i=1}^n \lambda_i$. Then $\lambda-\gamma(c\epsilon_0)$
is $\sl_{n+1}$-dominant integral. By the result in \cite{BG},
 $$T_V^{\gamma(c\epsilon_0),\lambda}:\ch_{\b1}^{\gamma(c\epsilon_0)}\rightarrow \ch_\b1^{\lambda}$$ is an equivalence, where $V$ is the finite dimensional simple $\sl_{n+1}$-module with the highest weight $\lambda-\gamma(c\epsilon_0)$.

 (3) By (1), there is some $k\in\{1,\dots,n\}$ such that $-\sum_{i=1}^n \lambda_i=\lambda_k-k$.Then $\lambda-\gamma(-k\epsilon_0)$
is $\sl_{n+1}$-integral. Hence $\lambda-\gamma(-k\epsilon_0)$
is in the $S_{n+1}$-orbit of the highest weight of some finite-dimensional simple $\sl_{n+1}$-module $V$.
Consequently claim (3) follows from the similar proof in (2).

(4) If $\lambda$ is $\sl_{n+1}$-integral and regular, then
$-\sum_{i=1}^n\lambda_i>\lambda_1-1$ or
there is
$j\in\{1,\dots,n\}$ such that
$\lambda_j-j>-\sum_{i=1}^n\lambda_i> \lambda_{j+1}-j-1$.
Take $\nu=\lambda$ or  $(j\cdots 0)\cdot \lambda$.
Then $\nu$ is $\sl_{n+1}$-dominant integral.
Therefore $T_{V(\nu)}^{0, \nu}:  \ch_\b1^{0} \rightarrow \ch_{\b1}^\nu=\ch_{\b1}^{\lambda}$
 is an equivalence.
\end{proof}

\subsection{Explicit constructions of simple modules in $\ch_{\b1}$}
Let $A_{n}= \C [x_1, \dots, x_{n}]$ be the polynomial algebra in $n$ variables.
Then the subalgebra of $\text{End}_{\C}(A_{n})$  generated by $\{x_i,\frac{\partial}{\partial x_i}\mid 1\leq i\leq n\}  $ is called the Weyl algebra $\cd_{n}$ over $A_{n}$.
In this subsection,  we will construct and classify simple modules in $\ch_{\b1}$ from $\cd_n$-modules and $\gl_n$-modules. It is known that there is
a Lie algebra homomorphism from
$\sl_{n+1}$ to the Witt algebra $W^+_n=\text{Der}(A_n)$  such that
$$\aligned
h_k& \mapsto x_k\frac{\partial}{\partial x_k}, &
e_{ij}& \mapsto x_i\frac{\partial}{\partial x_j},\ i\neq j,\\
e_{0j}&\mapsto -\frac{\partial}{\partial x_j},&
e_{i0}&\mapsto
x_i\sum_{q=1}^nx_q\frac{\partial}{\partial x_q},
\endaligned$$ where $1\leq i, j, k \leq n$, see page 578 in \cite{M}.

Then from the constructions of Shen-Larsson modules over  $W^+_n$,  see \cite{Sh,LLZ}, we have the following algebra homomorphism.

\begin{proposition}\label{s-iso} For any  $n\in \Z_{\geq 2}$, the linear map $\psi$ satisfying
\begin{equation}\label{3.1}\aligned U(\sl_{n+1})& \rightarrow \cd_{n}\otimes U(\gl_n),\\
h_k& \mapsto x_k\frac{\partial}{\partial x_k}\otimes 1+1\otimes E_{kk},\\
e_{ij}& \mapsto 1\otimes E_{ij}+x_i\frac{\partial}{\partial x_j}\otimes 1,\\
e_{0j}&\mapsto -\frac{\partial}{\partial x_j}\otimes 1,\\
e_{i0}&\mapsto  \sum_{q=1}^nx_q\otimes E_{iq}+
x_i\sum_{q=1}^nx_q\frac{\partial}{\partial x_q}\otimes 1+ x_i\otimes I_n,
\endaligned \end{equation}  can define an associative algebra homomorphism, where $i,j,k \in \{1,\dots,n\}$ with $i\neq j$, and $I_n:=\sum_{i=1}^n E_{ii}$ is the identity matrix in $\gl_n$.
\end{proposition}

For a finite dimensional simple $\gl_n$-module $V$ and a simple module $P$ over $D_n$,  by $T(P, V ) $ we denote the space $P \otimes  V$
considered as a module over $ \sl_{n+1}$ through the homomorphism $\psi$.

For convenience, denote $d_i= x_i\frac{\partial}{\partial x_i}$. Let $\Omega:=D_n/\sum_{i=1}^nD_n(\frac{\partial}{\partial x_i}+1)$ which is a simple $D_n$-module. As a vector space,
$\Omega=\C[d_1,\dots,d_n]$.  The $D_n$-module structure
on $\Omega$ is  defined by
$$\aligned \frac{\partial}{\partial x_i}\cdot f(d)& =-f(d+e_i),\\
x_i\cdot f(d)& =-f(d-e_i)d_i,
\endaligned $$ where $f(d)=f(d_1,\dots,d_n)$, $f(d+e_i)=f(d_1,\dots,d_i+1,\dots, d_n)$, $i=1,\dots, n$.
It is easy to see that $T(\Omega, V ) \in \ch_{\b1}$ and
$\mathrm{wh}_{\b1}(T(\Omega, V ))=V$.

We identify  $\lambda=(\lambda_1,\dots, \lambda_n)\in \C^n$
with the linear map $\lambda$ on $(\oplus_{i=1}^n \C E_{ii})$
such that $\lambda(E_{ii})=\lambda_i$ for any $i$. For   $\lambda\in \C^n$,
let $V(\lambda)$ be the highest weight $\gl_n$-module
with the highest weight $\lambda$. It is known that $V(\lambda)$ is finite dimensional if and only if
$\lambda\in \Lambda_n^{+}:=\{\lambda\in \C^n\mid
\lambda_1-\lambda_2,\dots,\lambda_{n-1}-\lambda_n\in \Z_{\geq 0}\}$.

\begin{lemma}\label{tensor-mod}\begin{enumerate}[$($a$)$]
 \item For  $\lambda\in  \Lambda_n^{+}$,  $ T(\Omega,  V(\lambda) ) \in \ch_{\b1}^{\lambda}$, where $\lambda$ is identified with
    the weight $-(\sum_{i=1}^n\lambda_i)\epsilon_0+\sum_{i=1}^n\lambda_i\epsilon_i$ in $\mh_n^*$.
 \item  If  $\lambda$ is not $\sl_{n+1}$-regular integral,  then $T(\Omega,  V(\lambda))$ is a
    simple  $\sl_{n+1}$-module.
\end{enumerate}
\end{lemma}

\begin{proof}(a) We can see that $T(A_n, V(\lambda) ) $ is
the highest weight $\sl_{n+1}$-module with the highest weight
$\lambda$. So $T(A_n, V(\lambda)) \in \mathcal{M}^{\lambda}$.
Since both $T(A_n, V(\lambda))$ and $ T(\Omega, V )$ are
constructed through the same map
$\psi$ defined by (\ref{3.1}), they have the same generalized
central character.

(b) This claim follows  from  Theorem 4.7 in \cite{GN}.
\end{proof}

For
$\sl_{n+1}$-regular integral weights, by (4) in Lemma \ref{equ-lemm},  we only need to consider the principal block $\ch_{\b1}^{0}$.
We use $\delta_i\in(\oplus_{i=1}^n \C E_{ii})^*$, $i=1,\dots,n$, to denote the $i$-th fundamental weight, i.e., $\delta_i(E_{11})=\cdots= \delta_i(E_{ii})=1$ and $\delta_i(E_{jj})=0$ for $j>i$.
 It is
well-known that  $V(\delta_1)$  is isomorphic to the
natural representation of $\gl_n$ on $\mathbb{C}^n$ via the matrix
product. The exterior
power $\bigwedge^k(\mathbb{C}^n)=\mathbb{C}^n\wedge\cdots\wedge
\mathbb{C}^n\ \ (k\ \mbox{times})$ is a $\mathfrak{gl}_n$-module
with the action given by $$X(v_1\wedge\cdots\wedge
v_k)=\sum\limits_{i=1}^k v_1\wedge\cdots\wedge v_{i-1}\wedge
Xv_i\cdots\wedge v_k, \,\,\forall \,\, X\in \gl_n.$$ Moreover,
\begin{equation}\label{dk}{\bigwedge}^k(\mathbb{C}^n)\cong V(\delta_k),\,\forall\,\, 1\leq k\leq n,\end{equation}
 as
$\gl_n$-modules, see Exercise 21.11 in \cite{H}).
For convenience, let $V(\delta_0)=\bigwedge^0(\mathbb{C}^n)=\C$ be the 1-dimensional trivial $\gl_n$-module and  set
$v\wedge a=av$ for any $v\in\C^n, a\in\C$. One can see that
 $T(\Omega,\bigwedge^i(\mathbb{C}^n))\in \ch_{\b1}^{0}$ for any $i\in\{0,\dots,n\}$.

By Lemma 3.2 in \cite{LLZ}, for each $k\in \{0,1,\dots, n-1\}$, any $D_n$-module $P$, there is an $\sl_{n+1}$-module homomorphism
$$\aligned \pi_{k}: T(P,\bigwedge^k(\mathbb{C}^n))& \rightarrow T(P,\bigwedge^{k+1}(\mathbb{C}^n)), \\
f\otimes v&\mapsto \sum_{i=1}^n (\frac{\partial}{\partial x_i}\cdot f)\otimes (e_i \wedge v), \endaligned$$
such that $\pi_{k+1}\pi_{k}=0$, where $\{e_1,\dots, e_n\}$ is the standard basis of $\mathbb{C}^n$, $v\in\bigwedge^k(\mathbb{C}^n), f\in P$. For every $0\leq k\leq n-1$, denote   $\text{im}\pi_{k}$ by $L_{k+1}(P)$.

Since $e_{0i}-1$ acts locally nilpotently on $T(\Omega,V)$, $e_{0i}$ acts bijectively on $T(\Omega,V)$ for all $i\in\{1,\dots,n\}$.
So $T(\Omega,V)$ can be extended to a $U_S$-module.
Recall that $U_S=B\otimes W$. Restricted to $W$, $T(\Omega,V)$ has a $W$-module structure.
From $[e_{0i}, W]=0$, $W\text{wh}_{\b1}T(\Omega,V)\subset \text{wh}_{\b1}T(\Omega,V)$. So
$\text{wh}_{\b1}T(\Omega,V)=V$ is a $W$-module.

The following lemma
gives the explicit action of $W$ on a $\gl_n$-module $V$,  which is useful for the classification
of finite dimensional simple $W$-modules in terms of $\gl_n$-modules.
In $T(\Omega,V)$, we identify $1\otimes V$ with $V$.

\begin{lemma}\label{w-action} Let $V$ be a $\gl_n$-module.
 The action of $W$ on $\emph{wh}_{\b1}T(\Omega,V)=V$ is described as follows:
\begin{equation}\label{x-w-act}\aligned x_{ij}v &=  (E_{ij}-E_{ii})v,\ 1\leq i\neq j\leq n,\\
 \omega_k v
&=\sum_{j=1}^n (E_{kj}E_{jj}-E_{kj})v,\ 1\leq k\leq n,
\endaligned \end{equation}where $ v\in V$.
\end{lemma}

\begin{proof} From (\ref{3.1}), the action of $e_{kj}$ on $T(\Omega,V(\delta_k))$ is defined by
$$e_{kj}(f(d)\otimes v)=f(d+e_j-e_k)d_k\otimes v+f(d)\otimes E_{kj}v, $$
where $f(d)\in \Omega, v\in V(\delta_k)$.

From
 $$\aligned & \sum_{j=1}^ne_{kj}(d_j\otimes v+1\otimes E_{jj}v-1\otimes v)\\
 &= \sum_{j=1}^n (d_j-\delta_{jk})d_k\otimes v+ \sum_{j=1}^n (d_j-1)\otimes E_{kj}v\\
 &\ \ \ \ + \sum_{j=1}^n d_k\otimes E_{jj}v+\sum_{j=1}^n 1\otimes E_{kj}E_{jj}v,
  \endaligned$$
 we can compute that
 $$\aligned \omega_k(1\otimes v)
&=e_{k0}e_{0k}+\sum_{j=1}^ne_{kj}(h_j-1)e_{0k}e_{0j}^{-1}(1\otimes v)\\
&= ( -\sum_{q=1}^n d_q\otimes E_{kq}v-
d_k\sum_{q=1}^n(d_q-\delta_{qk})\otimes v- d_k\otimes I_n v)\\
& \ \ \ \
+ \sum_{j=1}^n(d_j-\delta_{kj})d_k\otimes v+\sum_{j=1}^n(d_j-1)\otimes E_{kj}v\\
&\ \ \ \ +\sum_{j=1}^n d_k\otimes E_{jj} v+ \sum_{j=1}^n 1\otimes E_{kj}E_{jj}v\\
&= \sum_{j=1}^n 1\otimes (E_{kj}E_{jj}-E_{kj})v,\  1\leq k\leq n.
\endaligned $$

Moreover, for any $i,j: 1\leq i\neq  j \leq n$,
 $$\aligned x_{ij}(1\otimes v)& = (e_{ij}e_{0i}e_{0j}^{-1}-h_i)(1\otimes v)\\
&= 1\otimes E_{ij}v+d_i\otimes v-(1\otimes E_{ii}v+d_i\otimes v)\\
&= 1\otimes (E_{ij}-E_{ii})v.\endaligned $$

This completes the proof.
\end{proof}

In the following  Theorem, we classify all simple objects in $\ch_{\b1}^{0}$.

\begin{theorem}\label{pi-image}
In the principal block
$\ch_{\b1}^{0}$, we have the following results.
\begin{enumerate}
\item  $T(\Omega,V(\delta_0))$ and $T(\Omega,V(\delta_n))$ are simple $\sl_{n+1}$-modules.
\item  For each $k\in \{0,1,\dots, n-1\}$, $L_{k+1}(\Omega)=\text{im}\pi_{k}=\text{ker} \pi_{ k+1}$ is a simple $\sl_{n+1}$-module.
    \item If $k\neq l\in \{0,1,\dots, n-1\}$, then $L_{k+1}(\Omega)\not\cong L_{l+1}(\Omega)$ .
\item $\emph{wh}_{\b1}(L_{1}(\Omega)),\cdots, \emph{wh}_{\b1}(L_{n}(\Omega))$ are all simple objects in the principal block $W^0\text{-mod}$ of
            $W\text{-mod}$, where the action of $W$ on
        $\emph{wh}_{\b1}(L_{i}(\Omega))$  is given by (\ref{x-w-act}).
\item The block
$\ch_{\b1}^{0}$ has $n$ simple objects: $L_{1}(\Omega),\cdots, L_{n}(\Omega)$.
\end{enumerate}
\end{theorem}

\begin{proof} By Proposition \ref{equ-hw}, to show the simplicity of a module $M\in
\ch_{\b1}^{0}$, it remains to prove that $\text{wh}_{\b1}(M)$
is a simple $W$-module.

(1) Since $\dim\text{wh}_{\b1} T(\Omega,V(\delta_0))=\dim\text{wh}_{\b1} T(\Omega,V(\delta_n))=1$, the $W$-modules  $\text{wh}_{\b1} T(\Omega,V(\delta_0)), \text{wh}_{\b1} T(\Omega,V(\delta_n))$ are simple.

(2) By Lemma 3.4 in \cite{LLZ}, $\sl_{n+1}(\text{ker} \pi_{k+1}/\text{im}\pi_k)=0$. If $\text{ker} \pi_{k+1}/\text{im}\pi_k\neq 0$, then by the definition of $\ch_{\b1}$, there is a nonzero $v\in \text{ker} \pi_{k+1}/\text{im}\pi_k$ such that $e_{0i}v=v\neq 0$ for all $i$, a contradiction.
 So $\text{ker} \pi_{k+1}=\text{im}\pi_k$.

By the definition of $\pi_k$, we can see that $$\text{wh}_{\b1}(\text{im}\pi_k)=\text{Span}\{ (\sum_{i=1}^n e_i)\wedge v \mid v\in  V(\delta_k)\}. $$

Let $$V_{k}=\text{Span}\{  (\sum_{i=1}^n e_i)\wedge e_{j_1}\wedge\dots \wedge e_{j_{k-1}}\wedge e_{j_k}\mid 1\leq j_1,\dots, j_k\leq n-1\}, $$ which is a subspace of $\text{wh}_{\b1}(\text{im}\pi_k)$.
From
\begin{equation}\label{e-in}(\sum_{i=1}^n e_i)\wedge e_{j_1}\wedge\dots \wedge e_{j_{k-1}}\wedge (\sum_{i=1}^n e_i)=0,\end{equation}
 we have  $ (\sum_{i=1}^n e_i)\wedge e_{j_1}\wedge\dots \wedge e_{j_{k-1}}\wedge e_n\in V_{k}$.
Hence as a vector space, $$\text{wh}_{\b1}(\text{im}\pi_k)=V_{k}\cong \bigwedge^k (\C^{n-1}).$$

Let $\mathfrak{a}_{n-1}=\text{Span}\{E_{ij}-E_{ii}\mid 1\leq i\leq n-1, 1\leq j\leq n\}$. The map $$\mathfrak{a}_{n-1}\rightarrow \gl_{n-1},\ \ \tilde{E}_{ij}: = E_{ij}-E_{in}\mapsto E_{ij},\ \forall\ 1\leq i,j\leq n-1,$$ is a Lie algebra isomorphism. Let
$$\tilde{e}_{j_1}\wedge\dots \wedge \tilde{e}_{j_k} :=(\sum_{i=1}^n e_i)\wedge e_{j_1}\wedge\dots \wedge e_{j_k}, 1\leq j_1,\dots, j_k\leq n-1.$$
Then we can check that
\begin{equation}\label{a-mod}\aligned &\tilde{E}_{lq}(\tilde{e}_{j_1}\wedge\dots \wedge \tilde{e}_{j_k})\\
=&\begin{cases}
 \tilde{e}_{j_1}\wedge\dots \wedge \tilde{e}_{j_{p-1}}\wedge \tilde{e}_{l}\wedge\dots \wedge\tilde{e}_{j_k},\  \text{if}\  q=j_p\ \text{for some}\ p\in\{1,\dots, k\};\\
 0,\  \text{otherwise},
\end{cases}\endaligned
\end{equation}
where $1\leq l,q\leq n-1$.

So the $\mathfrak{a}_{n-1}$-module $\text{wh}_{\b1}(\text{im}\pi_k)$ is actually isomorphic to
 the simple $\gl_{n-1}$-module $\bigwedge^k (\C^{n-1})$. By Lemma \ref{w-action}, the action of $x_{ij}$ on $\text{wh}_{\b1}(\text{im}\pi_k)$ is given by $E_{ij}-E_{ii}$.
Hence $\text{wh}_{\b1}(\text{im}\pi_k)$ is a simple $W$-module.

(3) Note that $\bigwedge^k (\C^{n-1})\ncong \bigwedge^l (\C^{n-1})$ as $\gl_{n-1}$-modules for $k \neq l $.

(4) By Theorem 1.2 in \cite{Pr2}, there is a bijection between isoclasses of simple objects in $W^0\text{-mod}$
 and the primitive ideals $I$ of $U(\sl_{n+1})$ such that $I \cap Z(U(\sl_{n+1})) = \ker \chi_0$  and $\text{Var}(I)=\overline{\mathcal{O}}_{e}$, where $\text{Var}(I)$ is the associated variety of $I$. By Lemma 5.1 in \cite{Pe}, the number of such primitive ideals is $n$. Then (4) follows from (3).

(5)  By the equivalence between $\ch_{\b1}^{0}$ and $W^0\text{-mod}$,  $L_{1}(\Omega),\cdots, L_{n}(\Omega)$ exhaust all simple modules in $\ch_{\b1}^{0}$, up to isomorphisms.

\end{proof}

Next we will classify simple modules in $\ch_{\b1}^\lambda$ for which $\lambda$ is not $\sl_{n+1}$-regular integral. By Lemma
\ref{equ-lemm}, we can assume that $\lambda=\gamma(c\epsilon_0)$ for some $c\in \C\backslash \Z$ or $c\in\{-1,\dots,-n\}$.
Let $\phi: U(\sl_{n+1})\rightarrow \cd_{n+1}$ be the algebra homomorphism such that
$\phi(e_{ij})=x_i\frac{\partial}{\partial x_j}$, for all  $i,j: 0\leq i\neq j\leq n$. Let $E_{n+1}=\sum_{i=0}^{n}x_i\frac{\partial}{\partial x_i}$ which is called the Euler vector field. It can be seen that
the image $$\text{im}\phi \subseteq D_{n+1}^{E_{n+1}}:=\{v\in \cd_{n+1}\mid [E_{n+1}, v]=0\}.$$ Using the lift by $\phi$, any
$D_{n+1}^{E_{n+1}}$-module $N$ becomes an $\sl_{n+1}$-module denoted by $N^{\phi}$.

\begin{lemma}\label{ann-lemm} Suppose that $c\in \C\backslash \Z$ or $c\in\{-1,\dots, -n\}$.
If $M$ is a module in $\ch_{\b1}^{\gamma (c\epsilon_0)}$, then $(\ker\phi) M=0$.
\end{lemma}
\begin{proof} We use the weighting functor $\mathfrak{W}$ defined by (\ref{3.3}).
By Lemma \ref{free}, $M$ is a free $U(\mh_n)$-module of finite rank. Choose  $u\in \C^n$ such that $\mathcal{B}^{\gamma (c\epsilon_0)}_{u}$ is not empty. Then
$\mathfrak{W}^u(M)\in \mathcal{B}_u^{\gamma (c\epsilon_0)}$.
By Theorem 8.7 in \cite{GS2}, $(\ker \phi) (\mathfrak{W}^{u}(M))=0$. So $(\ker \phi) M\subset I_{\alpha +u}M$ for any $\alpha\in \Z^n$. Since $M$ is a free $U(\mh_n)$-module of finite rank, we have that
$\cap_{\alpha\in \Z^n}(I_{\alpha+u} M)=0$. So $(\ker \phi) M=0$.
\end{proof}

\begin{proposition}\label{non-int-sing}Suppose that $c\in \C\backslash \Z$ or $c\in\{-1,\dots,-n\}$.
\begin{enumerate}[$($a$)$]
\item If $N$ is a $D_{n+1}^{E_{n+1}}$-module such that
$E_{n+1}$ acts as a scalar $c$, then the $\sl_{n+1}$-module $N^{\phi}$ has the central character $\chi_{c\gamma(\epsilon_0)}$.
 \item
If $M$ is a simple  module in $\ch_{\b1}^{\gamma(c\epsilon_0)}$,  then $M\cong T(\Omega, V_{-\frac{c}{n+1}})$,  where
 $V_{\frac{-c}{n+1}}$ is the one dimensional $\gl_n$-module defined by the
 linear map:
 $$E_{ij}\mapsto \frac{-\delta_{ij}c}{n+1}, \ \ i,j=1,\dots, n.$$
\item If $V$ is a simple  module in $W^{\gamma (c\epsilon_0)}\text{-mod}$,  then $V$ is isomorphic to the one-dimensional $W$-module $V'_{-\frac{c}{n+1}}$ defined by the map:
\begin{equation}\label{one-dim-w-module}x_{ij}\mapsto \frac{c}{n+1},\ \omega_k \mapsto \frac{c^2}{(n+1)^2}+\frac{c}{n+1}, \ i\neq j, k =1,\dots, n.\end{equation}
\end{enumerate}
\end{proposition}

\begin{proof}
(a) Since $\phi(Z(U))\subset Z(D_{n+1}^{E_{n+1}})=\C[E_{n+1}]$, $N^{\phi}$ have the same central character for all $D_{n+1}^{E_{n+1}}$-modules $N$ such that
$E_{n+1}$ act as the same scalar $c$. Let $$F_{c}=\{ f(x)\in
x_0^{c}\C[x_0^{\pm 1},\dots, x_n^{\pm 1}]\mid E_{n+1}f(x)= cf(x) \}.$$
Then the $\sl_{n+1}$-submodule of $F_{c}^{\phi}$  generated by $x_0^{c}$ is a highest weight  module of
 $\sl_{n+1}$ with the highest weight $\gamma(c\epsilon_0)$.

(b) Let $I_{\b1}$ be the left ideal of $D_{n+1}$
generated by $$x_0-1, \frac{\partial}{\partial x_1}-1,\dots, \frac{\partial}{\partial x_n}-1.$$
Then we have the $D_{n+1}$-module $\Omega(\b1)=D_{n+1}/I_{\b1}$. Consider the $D_{n+1}^{E_{n+1}}$-module
$\Omega(\b1,c)=\Omega(\b1)/(E_{n+1}-c)\Omega(\b1)$.

 As a vector space
 $\Omega(\b1, c)=\C[d_{1},\dots, d_{n}]$. The action of $D_{n+1}^{E_{n+1}}$ of  $\Omega(\b1,c)$
 is defined by
 $$\aligned \ x_i\frac{\partial}{\partial x_j}f(d)&=f(d+e_j-e_i)d_i, \\
 x_0\frac{\partial}{\partial x_0}f(d)& =(c-d_1-\dots-d_n)f(d),\\
 x_0\frac{\partial}{\partial x_i}f(d)& =f(d+e_i), \\
 x_i\frac{\partial}{\partial x_0}f(d)& =(c-d_1-\dots-d_n+1)f(d-e_i)d_i, \\
 \endaligned$$ where $1\leq i,j\leq n, f(d)\in \Omega(\b1, c)$.

 Consequently, $\sl_{n+1}$ acts on $\Omega^{\phi}(\b1, c)$ as follows:
  $$\aligned e_{ij}f(d)&=f(d+e_j-e_i)d_i, \\
  h_if(d)& =(d_i-\frac{c}{n+1})f(d), \\
 e_{i0}f(d)& =(c-d_1-\dots-d_n+1)f(d-e_i)d_i, \\
 \endaligned$$ where $1\leq i \neq j\leq n, f(d)\in \Omega(\b1, c)$.

Then it can be checked that $\Omega^{\phi}(\b1, c)\cong T(\Omega, V_{\frac{-c}{n+1}})$.
 By (b) in Lemma \ref{tensor-mod},  the $\sl_{n+1}$-module $T(\Omega, V_{\frac{-c}{n+1}})$ is simple when  $c\in \C\backslash \Z$ or $c\in\{-1,\dots,-n\}$. Then Claim (b) follows from Lemma \ref{ann-lemm}.

(c) By Corollary \ref{p7} and Lemma \ref{w-action}, $V\cong \text{wh}_{\b1}(T(\Omega, V_{-\frac{c}{n+1}}))\cong V'_{-\frac{c}{n+1}}$.
\end{proof}

\subsection{ The equivalence between $\ch_{\b1}^{\lambda}$ and $\mathcal{C}_{\mu}^\lambda$}

  Let $\mu\in\C^n, \lambda\in\mh_n^*$ such that  $\mathcal{C}^\lambda_{\mu} $ is nonempty.  We define the functor $F_1: \mathcal{C}^\lambda_{\mu} \rightarrow W^\lambda\text{-mod}$ such that $F_1(M)=M_{\mu}$. From $[W, \mh_n]=0$,   we have $W M_{\mu}\subset M_{\mu}$. So $M_\mu \in W\text{-mod}$.
 Conversely, for  a $V\in  W\text{-mod}$, let each $h_{i}$
 act  on it as the scalar $\mu_i$.
 Let $$G_1(V)=\text{Ind}_{U(\mh_n)W}^{U_S}V=U_{S}\otimes_{U(\mh_n)W} V.$$ It is clear that $G_1(V)=\C[e_{01}^{\pm 1},\dots, e_{0n}^{\pm 1}]\otimes V$.
  From  (\ref{ws-relation}),
   the $\sl_{n+1}$-module structure on $G_1(V):=\C[e_{01}^{\pm 1},\dots, e_{0n}^{\pm 1}]\otimes V$ satisfies:
  $$\aligned h_k \cdot (e^{\br}\otimes v)& =(\mu_k-r_k)(e^{\br}\otimes v),\\
  e_{0k} \cdot (e^{\br}\otimes v)& =e^{\br+e_k}\otimes v,\\
    \endaligned$$
      $$\aligned
   e_{lj}\cdot  (e^{\br}\otimes v)& =e^{\br-e_l+e_j}\otimes x_{lj} v+(\mu_l-r_l+1)e^{\br-e_l+e_j}\otimes v,\ l\neq j,\\
  e_{l0}\cdot (e^{\br}\otimes v)
  &= e^{\br-e_l}\otimes\Big( \omega_l v-\sum_{j=1,\atop j\neq l}^n(\mu_j-r_j) x_{lj}v- (|\mu|-|\br|)(\mu_l-r_l+1) v\Big),
  \endaligned$$
where $ e^{\br}=e_{01}^{r_1}\dots e_{0n}^{r_n}, \br=(r_1,\dots, r_n)\in \Z^n, 1\leq l, j, k\leq n$. Thus we have a functor $G_1$ from
$W\text{-mod}$ to the category of weight modules over $\sl_{n+1}$. In the following theorem, we show that  each  $W^\lambda\text{-mod}$  is  equivalent to a block of the cuspidal category.

\begin{theorem}
\label{equ-cw}Let $ \lambda\in\mh_n^*$. Then there is
$\mu\in\C^n$ such that
 \begin{enumerate}[$($a$)$]
 \item   $G_1(V)\in \mathcal{C}_{\mu}^\lambda$, for any $V\in W^\lambda\text{-mod}$.
 \item $W^\lambda\text{-mod}$   is equivalent to $\mathcal{C}_{\mu}^\lambda$.
 \end{enumerate}
\end{theorem}

\begin{proof}
(a)
By Lemma \ref{equ-lemm}, it suffices to consider that $\lambda=0$ or $\lambda=\gamma (c\epsilon_0)$ with $c\in \C\backslash \Z$ or $c\in\{-1,\dots, -n\}$.

\noindent{\bf Case 1.} $\lambda=0$. We choose  $\mu\in\C^n$ satisfying that
\begin{equation}\label{no-int1}\mu_1,\dots, \mu_n,  -|\mu|\not\in \Z.\end{equation}
 Corollary 6.14 in \cite{GS1} implies that   $\mathcal{C}_{\mu}^0$ is nonempty.

By (4) in Proposition \ref{pi-image}, $\text{wh}_{\b1}(L_{1}(\Omega)),\dots, \text{wh}_{\b1}(L_{n}(\Omega))$ are all simple objects in $W^0\text{-mod}$, with the action of $W$  given by (\ref{x-w-act}).

To show that  $G_1(V)\in \mathcal{C}_{\mu}^0$, we need to explain  that
root vectors $e_{ij}, e_{k0}$ act injectively on $G_1(V)$.
 First we suppose that $V$ is simple, i.e., $V= \text{wh}_{\b1}(L_{k}(\Omega))$ for some $1\leq k\leq n$. From (\ref{x-w-act}) and $\omega_k \text{wh}_{\b1}(L_{k}(\Omega))=0$, the actions of $e_{lj}, e_{l0}$ on $G_1(\text{wh}_{\b1}(L_{k}(\Omega)))$ satisfy that
 $$\aligned e_{lj}\cdot  (e^{\br}\otimes v)
   &=e^{r-e_l+e_j}\otimes (E_{lj}-E_{ll}+\mu_l-r_l+1) v,\  l\neq j,\\
  e_{l0}\cdot (e^{\br}\otimes v)
  &= -e^{\br-e_l}\otimes\Big(\sum_{j=1}^n(\mu_j-r_j) (E_{lj}-E_{ll}) v +  (|\mu|-|\br|)(\mu_l-r_l+1)v\Big),
  \endaligned$$
where $v\in \text{wh}_{\b1}(L_{k}(\Omega)), 1\leq l, j, k\leq n$.


Note that $\text{wh}_{\b1}(L_{k}(\Omega))\subset \bigwedge^{k}\C^n$, and
 $\bigwedge^{k}\C^n$ is spanned by
$$e_{J}:=e_{j_1}\wedge\dots \wedge e_{j_k} , 1\leq j_1<\cdots < j_{k}\leq n,$$ where $J=(j_1,\dots, j_{k})\in \Z_{>0}^{k}$. Moreover,
$E_{ll}e_J=\delta_{l,J}e_J$, where $$\delta_{l,J}=\begin{cases}
1,\  \text{if}\  l=j_p\ \text{for some}\ p\in\{1,\dots, k\};\\
 0,\  \text{otherwise}.
\end{cases}
$$
For any nonzero element $v=\sum_{J\in \Gamma }a_Je_J$ in $V$, where $0\neq a_J\in \C$, $\Gamma$ is a finite subset of $\Z_{>0}^{k}$, we have that
$$ (E_{lj}-E_{ll}+\mu_l-r_l+1) v=\sum_{J\in \Gamma }a_JE_{lj}e_J
+ \sum_{J\in \Gamma} a_J (-\delta_{l,J}+\mu_l-r_l+1)e_J, \ \ l\neq j.$$
Note that $E_{lj}e_J=0$ or $e_{J'}$ for some $J'\neq J$, and $-\delta_{l,J}+\mu_l-r_l+1\neq 0$ for any $J\in \Gamma$. So $ (E_{lj}-E_{ll}+\mu_l-r_l+1) v\neq 0$, and hence $e_{lj}$ acts injectively on $G_1(L_{k}(\Omega))$, for any $1\leq l\neq j\leq n$.
Furthermore, we can see that
$$\aligned & \sum_{j=1}^n(\mu_j-r_j) (E_{lj}-E_{ll}) v +  (|\mu|-|\br|)(\mu_l-r_l+1)v\\
= & \sum_{J\in \Gamma} \sum_{j=1, j\neq l}^n(\mu_j-r_j) a_JE_{lj}e_J-(|\mu|-|\br|-\mu_l+r_l)\sum_{J\in \Gamma}a_JE_{ll}e_J\\
& + (|\mu|-|\br|)(\mu_l-r_l+1)\sum_{J\in \Gamma}a_Je_J\\
=& \sum_{J\in \Gamma} \sum_{j=1, j\neq l}^n(\mu_j-r_j) a_JE_{lj}e_J
+\sum_{J\in \Gamma}a_J(|\mu|-|\br|+\delta_{l,J})(\mu_l-r_l+1-\delta_{l,J})e_J.
\endaligned $$
Since $(|\mu|-|\br|+\delta_{l,J})(\mu_l-r_l+1-\delta_{l,J}) \neq 0$, we can also show that $e_{l0}$ acts injectively on $G_1(\text{wh}_{\b1}(L_{k}(\Omega)))$ for any $1\leq l\leq n$.
By induction on the length of the $W$-module $V$, we can show that
$G_1(V)$ is a cuspidal module for any $V\in W^0\text{-mod}$.

\noindent{\bf Case 2.}  $\lambda=\gamma(c\epsilon_0)$ with $c\in \C\backslash \Z$ or $c\in\{-1,\dots,-n\}$.

Choose  $\mu\in\C^n$ such that
\begin{equation}\label{no-int2}\frac{c}{n+1}+\mu_1, \cdots,\frac{c}{n+1}+\mu_n, \frac{c}{n+1}-|\mu|\not\in \Z.\end{equation}

By Corollary 2.6 in \cite{GS1}, $\mathcal{C}_{\mu}^{\gamma(c\epsilon_0)}$
is nonempty.
From Proposition \ref{non-int-sing},  the module $V'_{-\frac{c}{n+1}}$ defined by (\ref{one-dim-w-module}) is the unique  simple object
in $W^{\gamma (c\epsilon_0)}\text{-mod}$.
The actions of $e_{ij}, e_{k0}$ on $G_1(V'_{-\frac{c}{n+1}})$ are given by
$$\aligned
   e_{ij}\cdot  (e^{\br}\otimes v)& =(\frac{c}{n+1}+\mu_i-r_i+1)e^{\br-e_i+e_j}\otimes v,\ i\neq j,\\
  e_{k0}\cdot (e^{\br}\otimes v)& =(\frac{c}{n+1}-|\mu|+|\br|)(\frac{c}{n+1}+\mu_k-r_k+1) e^{r-e_k}\otimes v,
  \endaligned$$
  where $v\in V'_{-\frac{c}{n+1}}$.
The condition (\ref{no-int2}) forces that $e_{ij}, e_{k0}$ act injectively on $G_1(V'_{-\frac{c}{n+1}})$.  Again by induction on the length of $V$, we have that
$G_1(V)$ is a cuspidal module for any $V\in W^{\gamma(c\epsilon_0)}\text{-mod}$.

(b)
Since $F_1G_1(V)\cong V$ for any $V\in W\text{-mod} $,
 we have $F_1G_1\cong\text{id}$. For any $M\in \mathcal{C}_{\mu}^{\lambda}$,
 set $V=M_{\mu}$. Since each $e_{0i}$ acts bijectively on $M$, as a vector
 space $M=\C[e_{01}^{\pm 1},\dots, e_{0n}^{\pm 1}]\otimes V$.  By  (\ref{ws-relation}) and the isomorphism $U_{S}\cong W\otimes  B$,
we see that $M\cong G_1(V)=G_1F_1(M)$. So  $G_1F_1\cong\text{id}$,
and hence $F_1: \mathcal{C}^\lambda_{\mu} \rightarrow W^\lambda\text{-mod}$
is an equivalence.
\end{proof}


We have known that $\ch_{\b1}=\oplus_{\chi_{\lambda}\in\Theta} \ch_{\b1}^{\lambda}$.
By Theorem \ref{GS-th}, Corollary \ref{p7}, Proposition \ref{non-int-sing} and Theorem  \ref{equ-cw}, we have the
following characterizations of all blocks $\ch_{\b1}^{\lambda}$ of $\ch_{\b1}$.

\begin{theorem}\label{main-th} Let $\lambda\in \mh_n^*$
\begin{enumerate}[$($a$)$]
 \item If $\lambda$ is  singular or non-integral, then  $\ch^\lambda_{\b1}$ is equivalent to the category of
finite-dimensional modules over $\C[[x]]$.
 \item   If $\lambda$ is  regular integral, then $\ch^\lambda_1$  is equivalent to the category
of finite dimensional locally nilpotent modules over the quiver
$$
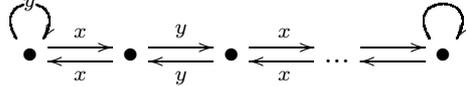
\xymatrix{\bullet \ar@(ul,ur)[]|{y} \ar@<0.5ex>[r]^x & \bullet
\ar@<0.5ex>[l]^x \ar@<0.5ex>[r]^y & \bullet \ar@<0.5ex>[l]^y
\ar@<0.5ex>[r]^x & \ar@<0.5ex>[l]^x... \ar@<0.5ex>[r] & \bullet
\ar@<0.5ex>[l] \ar@(ul,ur)[]|{}}
$$
containing $n$ vertices with relations $xy=yx=0$.
   \end{enumerate}
\end{theorem}

\subsection{The category $\ch_{\mathbf{0}}$}
For $\lambda\in \Lambda_n^{+}$,
let $V(\lambda)$ be the simple finite dimensional  $\ml_n$-module with the highest weight $\lambda$. Then $V(\lambda)$  can be extended to a module over the parabolic subalgebra $\mathfrak{p}_n:=\ml_n\oplus \mm_n$ by defining $\mm_n V(\lambda)=0$. The parabolic Verma module $M_{\mathfrak{p}_n}(\lambda)$ over $\sl_{n+1}$ is
$U(\sl_{n+1})\otimes_{U(\mathfrak{p}_n)}V(\lambda)$. Denote the unique simple  quotient of $M_{\mathfrak{p}_n}(\lambda)$ by $L_{\mathfrak{p}_n}(\lambda)$.

\begin{proposition}\label{h-zero} Any simple module $M$ in $\ch_{\mathbf{0}}$ is isomorphic to a simple quotient of some $T(A_n, V(\lambda))$, where $\lambda\in \Lambda_n^{+}$.
\end{proposition}
\begin{proof} Note that the subspace
 $\mathrm{wh}_{\mathbf{0}}(M)=\{v\in M \mid e_{0i} v=0, \ \forall\ i=1,\dots,n\}$ is finite dimensional. From $[\ml_n,  \mm_n]\subset \mm_n$,  we see that
 $\ml_n\mathrm{wh}_{\mathbf{0}}(M)\subset \mathrm{wh}_{\mathbf{0}}(M)$. The simplicity of $M$  implies that
 $\mathrm{wh}_{\mathbf{0}}(M)$ is a finite dimensional  simple  $\ml_n$-module. Since  $\ml_n\cong \gl_n$, $\mathrm{wh}_{\mathbf{0}}(M)\cong V(\lambda)$ for some  $\lambda\in \Lambda_n^{+}$. Since $\mm_n\mathrm{wh}_{\mathbf{0}}(M)=0$,
 $M$ is isomorphic to the simple quotient $L_{\mathfrak{p}_n}(\lambda)$ of parabolic Verma module $M_{\mathfrak{p}_n}(\lambda)$.  Then
 this proposition follows from that $L_{\mathfrak{p}_n}(\lambda)$  is a simple quotient of $T(A_n, V(\lambda))$.
\end{proof}

The parabolic BGG category $\mathcal{O}_{\mathfrak{p}_n}$ is the category  of finitely
generated weight $\sl_{n+1}$-modules that are locally finite over $\mathfrak{p}_n$. Clearly  $\mathcal{O}_{\mathfrak{p}_n}$ is a full subcategory of $\ch_0$. By Proposition \ref{h-zero}, $\mathcal{O}_{\mathfrak{p}_n}$ and $\ch_0$ have the same simple objects. There are many studies on $\mathcal{O}_{\mathfrak{p}_n}$. For example,  integral blocks of $\mathcal{O}_{\mathfrak{p}_n}$ were studied in \cite{BS}, injective modules in $\mathcal{O}_{\mathfrak{p}_n}$ were constructed in \cite{CG} in  terms of differential operators.

\vspace{2mm}
\noindent
{\bf Acknowledgments. }This research is supported  by NSF of China (Grants
12371026, 12271383).
The authors would like to thank
the referee for nice suggestions concerning the presentation of the
paper.

\vspace{4mm}

 \noindent G.L.: School of Mathematics and Statistics,
and  Institute of Contemporary Mathematics,
Henan University, Kaifeng 475004, China. Email: liugenqiang@henu.edu.cn

\vspace{0.2cm}

\noindent Y. Li: School of Mathematics and Statistics, Henan University, Kaifeng
475004, China. Email: 897981524@qq.com

\end{document}